\documentclass{amsart}
\usepackage{amsmath}
\usepackage{amssymb}
\usepackage[arrow,matrix,curve]{xy}
\usepackage{graphicx}

\newcommand\pt{\text{\textup{pt}}}

\newtheorem{theorem}{Theorem}[section]

\newtheorem{lemma}[theorem]{Lemma}
\newtheorem{proposition}[theorem]{Proposition}
\newtheorem{corollary}[theorem]{Corollary}

\theoremstyle{remark}

\newtheorem{remark}[theorem]{Remark}

\newtheorem{examples}[theorem]{Examples}

\newcommand\cA{\mathcal{A}}

\newcommand\cC{\mathcal{C}}

\newcommand\cH{\mathcal{H}}
\newcommand\cK{\mathcal{K}}

\newcommand\bC{\mathbb{C}}
\newcommand\bF{\mathbb{F}}
\newcommand\bK{\mathbb{K}}

\newcommand\bQ{\mathbb{Q}}
\newcommand\bR{\mathbb{R}}

\newcommand\bZ{\mathbb{Z}}
\newcommand\fp{\mathfrak{p}}

\newcommand\End{\operatorname{End}}

\newcommand\Cl{\operatorname{C\ell}}
\newcommand\Tor{\operatorname{Tor}}

\newcommand\Spec{\operatorname{Spec}}
\newcommand\coker{\operatorname{coker}}
\providecommand{\co}{\colon\,}
\newcommand\Ca{$C^*$-algebra}

\newcommand{\lp}{\textup{(}}
\newcommand{\rp}{\textup{)}}

\begin{document}
\title[The K\"unneth Theorem in Equivariant $K$-Theory]
{The K\"unneth Theorem in Equivariant $K$-Theory\\ 
for actions of a cyclic group of order $2$}
\author{Jonathan Rosenberg}
\address{Department of Mathematics,
University of Maryland,
College Park, MD 20742, USA}
\email{jmr@math.umd.edu}
\urladdr{http://www.math.umd.edu/\raisebox{-.6ex}{\symbol{"7E}}jmr}
\begin{abstract}
The K\"unneth Theorem for equivariant {\lp}complex{\rp} $K$-theory
$K^*_G$, in the form developed by Hodgkin and others, fails
dramatically when $G$ is a finite group, and even when $G$ is cyclic
of order $2$. We remedy this situation in this very simplest case
$G=\bZ/2$ by using the power of $RO(G)$-graded equivariant $K$-theory.
\end{abstract}
\maketitle

\section{Introduction}
\label{sec:intro}
Equivariant $K$-theory, invented by Atiyah and Segal (for the original
exposition, see \cite{MR0234452}), is the simplest equivariant
cohomology theory to define. It is enormously useful: in
equivariant topology, in index theory (where it is needed for the
equivariant index theorem),
and in the theory of operator algebras.  (If $X$
is a locally compact $G$-space, then $C_0(X)$, the algebra of
continuous functions on $X$ vanishing at infinity, is a $G$-{\Ca}, and
$K^{-*}_G(X) \cong K^G_*(C_0(G))$. Note that on $G$-algebras, equivariant
$K$-theory becomes a homology theory instead of a cohomology
theory. For the theory of equivariant $K$-theory for operator algebras,
see \cite[Ch.\ V, \S 11]{MR1656031}.)

Despite its apparent simplicity, equivariant $K$-theory is still quite
puzzling in many respects.  This is already evident when one studies
the K\"unneth Theorem, or in other words, when one attempts to compute
$K^*_G(X\times Y)$ given knowledge of $K^*_G(X)$ and $K^*_G(Y)$ (or
dually, to compute $K^G_*(A\otimes B)$ in terms of $K^G_*(A)$ and
$K^G_*(B)$).  The first, and still the most important, work on this
problem was done by Hodgkin \cite{MR0478156}. Hodgkin observed that
since the coefficient ring for $G$-equivariant $K$-theory is the
(complex) representation ring $R(G)$ of $G$, a K\"unneth Theorem for
$K^*_G$ should take the form of a spectral sequence
\begin{equation}
\label{eq:Hodgkin}
\Tor_p^{R(G)}(K^*_G(X), K^{-*+q}_G(Y)) \Rightarrow K^{p+q}_G(X \times Y),
\end{equation}
which he constructed.
However, Hodgkin noticed that there are two big problems with this:
\begin{enumerate}
\item If $G$ is a disconnected compact Lie group, then $R(G)$ never
  has finite homological dimension, and so this sequence can't be
  expected to converge.
\item Even if $G$ is connected, but if $\pi_1(G)$ is not torsion-free,
  then the spectral sequence may converge, but to the wrong limit.
\end{enumerate}
In particular, Hodgkin's theorem, which was improved a bit by
Snaith \cite{MR0309115} and McLeod \cite{MR557175}, is of no help at
all if $G$ is finite or if $\pi_1(G)$ has torsion. In joint work of
the author with Schochet \cite{MR849938}, we extended Hodgkin's
theorem to the {\Ca}ic case of $K^G_*(A\otimes B)$, with $A$ and $B$
nuclear $G$-algebras in a suitable ``bootstrap'' class (containing all
countable inductive limits of separable commutative $G$-{\Ca}s), but
again only for $G$ connected compact Lie
with $\pi_1(G)$ torsion-free. (We did,
however, manage to elucidate the meaning of the condition that
$\pi_1(G)$ be torsion-free. For connected compact Lie groups $G$, this
is equivalent to the condition that every action of $G$ on the compact
operators $\cK$ be exterior equivalent to a trivial action.)
Thus the ``puzzle'' of what should replace the K\"unneth Theorem when
$G$ is finite remained open.

The other major piece of work on this problem was done by Chris
Phillips \cite{MR911880}. He did address the K\"unneth Theorem for
equivariant $K$-theory for $G$ finite, but only obtained a partial
result, since he was relying on the Localization Theorem of Segal
\cite[Proposition 4.1]{MR0234452}. 

While in this paper we  sometimes work in the generality of group
actions on
{\Ca}s, the reader should realize that the case where the {\Ca}s are
commutative is highly non-trivial and already new, and those not
interested in operator algebras can restrict themselves to this case
without missing very much. However, generalizing to the noncommutative
case makes the proofs easier, since as first pointed out in
\cite{MR650021}, geometric resolutions are actually easier to
construct in the noncommutative world.

\section*{Acknowledgements}
This work was supported by NSF Grant DMS-1206159. It grew out of the
author's puzzlement about certain aspects of equivariant $K$-theory
that came up in joint work with Mathai Varghese \cite{MathaiRos}
begun in March, 2012,
with partial support from the University of Adelaide and the
Australian Mathematical Sciences Institute. Some of the ideas in this
paper go back to the author's joint work with Claude Schochet in the
1980's.

\section{Background, Notation, and Previous Results}
\label{sec:classical}
We begin by recalling some previous results and establishing
notation. In particular, we restate the results of Phillips 
\cite{MR911880} and Izumi \cite{MR2053753}
in terms a topologist would appreciate since it is likely that their
work is not known to most topologists interested in equivariant
$K$-theory. 

Throughout this paper, $K$-theory or equivariant $K$-theory for spaces
always means complex topological $K$-theory \emph{with compact
supports} for locally compact Hausdorff spaces. In most cases these
spaces will be second countable and thus paracompact.
Because of Bott periodicity, we will sometimes regard this theory as
being $\bZ/2$-graded. This theory satisfies a very
strong form of excision --- if $X$ is a closed $G$-subspace of $Y$,
then $K^*_G(Y, X) \cong K^*_G(Y\smallsetminus X)$. (Note that
$Y\smallsetminus X$ is indeed locally compact.)

From now on, let $G$ be a cyclic group of prime order $q$ and let
$R=R(G)$ be its representation ring, which we identify with
$\bZ[t]/(t^q -1)$. Here $t$ represents the standard representation of
$G$ on $\bC$ in which a fixed generator $g$ of $G$ is sent to
$\zeta=\exp(2\pi i/q)$. This ring is the coefficient ring for equivariant
$K$-theory. Its ideal structure was studied in \cite{MR0248277}.  Let
$I=(t-1)$ be the augmentation ideal and let $J=(1+t+\cdots +
t^{q-1})$. Since $(t-1)(1+t+\cdots + t^{q-1})=0$ in $R$, each prime
ideal $\fp$ of $R$ contains either $I$ or $J$, and these are the
unique minimal prime ideals of $R$ by \cite[Proposition
3.7]{MR0248277}. In the language of Segal, the prime ideal $I$ has
support $\{1\}$, while the prime ideal $J$ has support $G$. (Since
$\{1\}$ and $G$ are the only subgroups of $G$, these are the only two
possibilities.) Note that $G/I\cong \bZ$, while $G/J\cong \bZ[\zeta]$
is the ring of integers in the cyclotomic field $\bQ[\zeta]$. 
Similarly, the localizations of $R$ at these two prime ideals are
$R_I \cong \bQ$ and $R_J \cong \bQ[\zeta]$. Since $G/I$ and $G/J$ are
both Dedekind domains, the other prime ideals of
$R$ are all maximal ideals. If such a maximal ideal $\fp$ contains
$I$, then it is of the form $(I, p)$ for $p$ a prime of
$\bZ$ generating $(\fp\cap \bZ) \triangleleft \bZ$, 
while if it contains $J$, then it contains $(J, p)$ for some prime $p$.
The arithmetic of the cyclotomic field (the splitting of
primes $p$ in $\bZ[\zeta]$) will come in at this point when one 
classifies the prime (and necessarily maximal) ideals over $(J,p)$. 
There are now two cases: if $p=q$, then $R/(p)\cong \bF_q[t]/(t^q-1)
\cong \bF_q[u]/(u^q)$, where $u=t-1$. So $p$ ramifies in the
cyclotomic field and $(J, p) \subseteq (I, p)$, with $(I, p)$ the
unique maximal ideal of $R$ containing $p$. This ideal has support
$\{1\}$ in the sense of Segal. Otherwise, if $q\ne p$, the primes over
$(J, p)$ are distinct from the primes over $I$, and have support $G$
(cf.\ \cite[Proposition 6.2.2]{MR911880}).
In any event, the localization $R_{\fp}$ of $R$ at a maximal
ideal will be isomorphic to $\bZ_{(p)}$ if $\fp = (I, p)$ and to
a localization of $\bZ[\zeta]$ if $\fp \supseteq (J, p)$.
So $R_{\fp}$ is a discrete valuation ring and thus has
global dimension $1$. For purposes of this paper we will eventually
further restrict to the case $q=2$, in which case $J=(t+1)$ and 
$R/J$ is also isomorphic to $\bZ$, and the
maximal ideals of $R$ are precisely $(I, p)$ or $(J, p)$ for $p$ a
prime. In this case $(I, 2) = (J, 2)$ since $(t-1)+2 = t+1$; otherwise
the ideals $(I,p)$ and $(J, p)$ are all distinct. The
picture of $\Spec R$, showing the inclusion relations among prime
ideals, is shown in Figure \ref{fig:specR}. Note the left-right
reflection symmetry of the diagram, which can be explained by the
existence of an automorphism $t\mapsto -t$ of $R$ (quite special to
the case $q=2$) which interchanges $I$ and $J$.
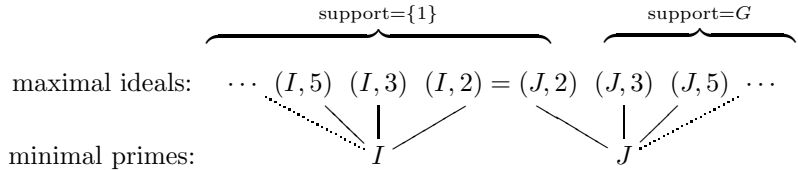
\begin{figure}[ht]
\[
\xymatrix@C-2pc@R-1pc{
& & & & \save*{\overbrace{\makebox[1.8in]{}}^{\text{support} = \{1\}}}\restore & & &
  \save*{\overbrace{\makebox[1in]{}}^{\text{support} = G}} \restore&\\ 
\text{maximal ideals:}& &
\cdots & (I, 5) & (I, 3) & (I, 2)=(J,2) 
& (J, 3) & (J, 5) & \cdots\\
\text{minimal primes:} & & &  & I \ar@{.}[ull] 
\ar@{-}[ul] \ar@{-}[u] \ar@{-}[ur] 
& & J \ar@{-}[ul] \ar@{-}[u] \ar@{-}[ur] \ar@{.}[urr]& & 
}
\]
\caption{A picture of $\Spec R$ for $G=\bZ/2$} 
\label{fig:specR}
\end{figure}

Before proceeding, it is convenient to recall the following
calculation of
equivariant $K$-theory for free actions, which almost certainly is
known to experts but is not explicit in \cite{MR0234452}.
\begin{proposition}
\label{prop:equivKfree}
Let $G$ be a cyclic group of order $q$
and let $X$ be a compact {\bfseries free} $G$-space. 
Then the $R$-module structure on $K^*_G(X)\cong K^*(X/G)$ is
defined by letting $t$ act by tensoring with the line bundle $V$
with $c_1(V)=c$, where $c$ is the image in $H^2(X/G, \bZ)$ under the
Bockstein homomorphism of the class in $H^1(X/G, \bF_q)$ classifying
the $q$-to-$1$  covering map $X\to X/G$. One can also realize $V$ more
explicitly as the fiber product $X\times_G \bC$, 
where $G$ acts on $\bC$ by the nontrivial character $t$.

If $A$ is a closed $G$-invariant subspace of $X$, then the
$R(G)$-module structure on 
\[
K^*_G(X,A)\cong K^*(Y, B), \quad Y=X/G,\, B=A/G,
\]
is again
defined by letting $t$ act by cup-product with $[V]\in
K^0(X/G)$. {\lp}Recall that for any pair $(Y,B)$, we have the cup-product
$K^0(Y)\otimes K^*(Y,B) \to K^*(Y,B)$.{\rp}
\end{proposition}
\begin{proof}
The definition of the $R(G)$-action on $K^*_G(X)$ or on $K^*_G(X,A)$
implies that the result of applying the module action of $t$
corresponds to tensoring with $(\bC, t)$, which is the same after
applying the isomorphism $K^*_G(X)\cong K^*(Y)$ or $K^*_G(X,A)\cong
K^*(Y,B)$ as the vector bundle tensor product with $V$. The rest is
immediate. 
\end{proof}

The following result of Izumi constrains $K_*^G(A)$ if $A$ is a
$G$-{\Ca}, or $K^*_G(X)$ if $X$ is a locally compact $G$-space, and if
$K_*(A)=0$ or $K^*(X)=0$ (in all degrees). 

\begin{theorem}[{Izumi, \cite[Lemma 4.4]{MR2053753}}]
\label{thm:Izumi}
Let $A$ be a $G$-{\Ca}, where $G$ is a cyclic group of prime order
$q$. Assume that $A$ is $K$-contractible, i.e., that
$K_*(A)=0$ {\lp}non-equivariantly, in both odd and
even degrees{\rp}. Then, as a
$\bZ$-module {\lp}i.e., forgetting the $R$-module structure{\rp},
$K_*^G(A)$ is uniquely $q$-divisible. Similarly, if $X$ is a locally
compact $G$-space and $K^*(X)=0$, then $K^*_G(X)$ is uniquely
$q$-divisible as a $\bZ$-module.
\end{theorem}

Note that Izumi phrased Theorem \ref{thm:Izumi}
in terms of the $K$-theory of the crossed
product $A\rtimes \bZ/q$, but this is the same as $K_*^G(A)$ because
of the Green-Julg Theorem (\cite{MR625361}, 
\cite[Theorem 11.7.1]{MR1656031}).  

The theorem cannot be improved, even in the abelian case, because if 
if $X$ is a locally compact $G$-space and $K^*(X)=0$, then $K^*_G(X)$
is \emph{not necessarily zero}. It was pointed out in \cite[Lemma 5.7]%
{MR2560910} that Lowell Jones's converse \cite{MR0295361}
to P.\ A.\ Smith's Theorem provides a counterexample. However, since
the proof there was slightly garbled (as pointed out by Thomas Schick
in the review in MathSciNet), we restate it again.
\begin{proposition}
\label{prop:KzeroKGnot}
Let $G$ be a cyclic group of prime order $q$. Then there is a
contractible finite $G$-CW complex $Y$ for which $L = Y^G$ has torsion
in its homology of order prime to $q$. We can choose a basepoint
$x_0\in L$ so that if $X= Y\smallsetminus \{x_0\}$, $K^*(X)=0$ while
$K^*_G(X)\ne 0$.
\end{proposition}
\begin{proof}
By \cite{MR0295361}, if $L$ is a finite CW-complex with $\widetilde
H^*(L, \bZ/q)=0$ (in all degrees), 
then we can choose a contractible finite $G$-CW
complex $Y$ with $L = Y^G$. Clearly $L$ can be chosen with a basepoint $x_0$
so that $K^*(L\smallsetminus \{x_0\})$ contains torsion of
order a prime $\ell\ne q$ (though $K^*(L\smallsetminus \{x_0\};
\bZ/q) = \widetilde K^*(L; \bZ/q)=0$). Since $Y$ is contractible, if
$X= Y\smallsetminus 
\{x_0\}$, then $K^*(X)=\widetilde K^*(Y)=0$. Choose a maximal ideal
$\fp$ of $R$ containing $(J, \ell)$. Then $\fp$ has support $G$ in the
sense of Segal, so by the Localization Theorem \cite[Proposition 4.1]%
{MR0234452}, $K^*_G(X)_{\fp}\cong K^*_G(L\smallsetminus \{x_0\})_{\fp}
\cong K^*(L\smallsetminus \{x_0\})\otimes_{\bZ}R_{\fp}$. This is
non-zero since $K^*(L\smallsetminus \{x_0\})\otimes_{\bZ} \bZ/\ell \ne 0$
and $R_{\fp}/{\fp}_{\fp}$ is a finite field of characteristic $\ell$.
Thus $K^*_G(X)_{\fp}\ne 0$ and $K^*_G(X)\ne 0$.
\end{proof}

Applying Takai Duality
\cite{MR0365160} to Theorem \ref{thm:Izumi}, 
we deduce the following. 
\begin{corollary}
\label{thm:dualIzumi}
Let $A$ be a $G$-{\Ca}, where $G$ is a cyclic group of prime order
$q$. If $K_*^G(A)=0$ {\lp}in all degrees{\rp}, then
$K_*(A)$ is uniquely $q$-divisible as a $\bZ$-module. 
Similarly, if $X$ is a locally
compact $G$-space and $K^*_G(X)=0$, then $K^*(X)$ is uniquely
$q$-divisible.
\end{corollary}

Now suppose one has a $G$-{\Ca} $B$ with $K_*^G(B)=0$ but
$K_*(B)\ne 0$. One can get such an example by starting with a
$K$-contractible $G$-{\Ca} for which $K_*^G(A)$ is non-zero, 
as provided by Proposition \ref{prop:KzeroKGnot}  or by \cite[Lemma
4.7]{MR2053753}.  Then
let $B=A\rtimes \bZ/q$ and consider the \emph{dual} action of $\bZ/q$
on $B$; then $B$ is \emph{not} $K$-contractible since $K_*(B)\cong
K_*^G(A) \ne 0$, but by Takai duality, one finds that 
$K_*^G(B) \cong K_*(A) = 0$. However, $K_*(B)\cong K^G_*(C(G)\otimes
B)$, so we 
find that \emph{the K\"unneth Theorem in equivariant $K$-theory fails}
for $B$, in the sense that $K_*^G(B)$ is identically $0$ but there is
a $G$-algebra (namely $C(G)$) for which the tensor product has
non-vanishing (but uniquely $q$-divisible)
equivariant $K$-theory.  (See \cite{MathaiRos} for more
details.) To sum up, knowing just $K_*^G(C)$ and $K_*^G(D)$, we
cannot hope for a spectral sequence computing $K_*^G(C\otimes D)$. The
fact that the (na\"\i ve) K\"unneth Theorem in equivariant $K$-theory
fails for actions of finite groups was already pointed out in
\cite[Example 6.6.9]{MR911880}. 

However, we can now state the K\"unneth Theorem of Phillips.
\begin{theorem}[{Phillips, \cite[Theorem 6.4.6]{MR911880}}]
\label{thm:Phillips}
Let $G$ be a cyclic group of prime order $q$, and let $\fp$ be a prime
ideal of $R=R(G)$ with support $G$. {\lp}Thus either $\fp=J$ or $\fp$
contains $(J, p)$ for $p$ a prime $\ne q$.{\rp} Let $A$ and $B$ be
separable
$G$-{\Ca}s with $B$ nuclear and with $A$ in a suitable bootstrap
category containing all equivariant inductive limits of separable type I
$G$-{\Ca}s and closed under various conditions {\lp}see 
\cite[Theorem 6.4.7]{MR911880} for more details{\rp}.
Then there is a functorial short exact sequence
\[
0 \to K_*^G(A)_{\fp} \otimes_{R_{\fp}} K_*^G(B)_{\fp}
\xrightarrow{\omega_{\fp}} K_*^G(A\otimes B)_{\fp} \to
\Tor_1^{R_{\fp}}\bigl(K_*^G(A)_{\fp}, K_{*+1}^G(B)_{\fp} \bigr) \to 0,
\]
which splits, though not naturally.  The theorem holds in particular
if $A=C_0(X)$ and $B=C_0(Y)$ with $X$ and $Y$ second countable locally
compact $G$-spaces.
\end{theorem}
\begin{remark}
\label{rem:HodgkinPhillips}
Note that Theorem \ref{thm:Phillips} is exactly what would expect from
a Hodgkin-type spectral sequence \eqref{eq:Hodgkin}, since
localization is an exact functor, and thus one would get a spectral
sequence
\begin{equation}
\label{eq:Phillips}
\Tor_p^{R(G)_{\fp}}(K^*_G(X)_{\fp}, K^{-*+q}_G(Y)_{\fp}) \Rightarrow
K^{p+q}_G(X \times Y)_{\fp}, 
\end{equation}
which would collapse at $E_2$, giving a short exact sequence, since
$R_{\fp}$ is a PID in this case, 
and thus all higher Tor's (beyond $\Tor_1$) vanish. Phillips's insight
is that \eqref{eq:Phillips} holds for the prime ideals 
mentioned in the theorem, even
though the Hodgkin-type spectral sequence \eqref{eq:Hodgkin} fails.
\end{remark}
\begin{remark}
\label{rem:Phillipsfails}
Note that Theorem \ref{thm:Phillips}
fails, even in the commutative case, if $\fp =
I$ or $\fp =(I, p)$, even though $R_{\fp}$ has global dimension $1$, so that
homological algebra alone is not the explanation. Indeed, note
that $K^*_G(G)\cong R/I$, which when localized at $\fp$ gives just
$R_{\fp}$, which is free of rank $1$ as an $R_{\fp}$-module.  Thus the
theorem, if true, would say that $K^*_G(G\times X)_{\fp} \cong
K^*(X)_{(p)}$ is always isomorphic to $K^*_G(X)_{\fp}$ at least as a
$\bZ_{(p)}$-module. But this is false even in the rather trivial case
of $X=G$ with the simply transitive $G$-action, 
since $K^*_G(G\times G)_{\fp} \cong R_{\fp}^q$ while 
$K^*_G(G)_{\fp} \cong R_{\fp}$.
\end{remark}

\section{A finer invariant}
\label{sec:ROgraded}
To deal with the failure of the K\"unneth Theorem, we need a finer
invariant than just equivariant $K$-theory alone.
An important fact about equivariant $K$-theory for a compact group
$G$ is that it is naturally \emph{$RO(G)$-graded} (see for example
\cite[Chapters IX, X, XIII, and XIV]{MR1413302}). Given a compact
group $G$, a locally compact $G$-space, and a finite-dimensional real
orthogonal representation $V$ of $G$, we can form $K^*_{G,V}(X) =
K^*(X\times V)$. Similarly, given a $G$-{\Ca} $A$, we can define
$K_*^{G,V}(A) = K_*^G(A \otimes C_0(V))$, where $G$ acts on the second
factor via the linear representation and acts on the tensor product by
the tensor product action.  Note that if $V$ happens to be
a complex vector space and the action of $G$ is complex linear, then
equivariant Bott periodicity \cite[Proposition 3.2]{MR0234452}
gives an isomorphism $K^*_{G,V} \cong K^*_G$ or $K_*^{G,V} \cong
K_*^G$. (This is also true more generally if $V$ is even-dimensional
over $\bR$ and if the action of $G$ factors through $\text{Spin}^c(V)$.)
And if $G$ acts trivially on $V$, $K_*^{G,V} \cong K_{*+\dim
  V}^G$. But in general the groups $K_*^{G,V}$ are not the same as
$K_*^G$, even modulo a grading shift. In the noncommutative world,
another approach to the groups $K_*^{G,V}$ is possible via graded
Clifford algebras, since
$C_0(V)$ is $KK^G$-equivalent to $\Cl(V)$, the complex Clifford
algebra of $V$ viewed as a \emph{graded} $G$-algebra \cite[Theorem
  20.3.2]{MR1656031}. But this requires introducing graded {\Ca}s,
which we'd prefer to avoid.

For the rest of the paper we will deal only with the case where
$G=\{1, g\}$ is cyclic of order $2$.\footnote{Unfortunately,
$RO(G)$-graded $K$-theory doesn't give anything new when $G$ is
cyclic of odd prime order, since then all non-trivial real irreducible
representations of $G$ are actually complex.}
This group $G$ has exactly 
two real characters, the trivial character $1$ and
the non-trivial character $t$, the sign representation $-$ (where the
generator $g$ of the group acts by $-1$ on $\bR$).  From the sign
representation we get the twisted
equivariant $K$-groups $K^*_{G, -}$ (on spaces) or $K_*^{G, -}$ (on
algebras). These are modules over the representation ring $R$.
The coefficient groups for $K^*_{G, -}$
are computed in \cite{MR2545608}, for example. It turns
$K^*_{G, -}(\pt)\cong R/J$, concentrated in even degree.
Twisting twice brings us back  
to conventional equivariant $K$-theory since a direct sum of two
copies of the sign character is a complex representation, where
equivariant Bott periodicity applies.

We can now define an invariant of a $G$-space (or $G$-{\Ca}) finer
than just the equivariant $K$-theory alone. Namely, note that if $V$
is $\bR$ with the sign representation of $G$, then we have 
a $G$-map $\{0\}\hookrightarrow V$ inducing (for any $G$-space)
a natural $R$-module homomorphism $\varphi\co K^*_{G,-}(X)\to
K^*_G(X)$, which when $X$ is a point can be identified with the
composite $R/J \cong I \hookrightarrow R$. Via the composite
\[
K^*_G(X) \xrightarrow{\text{equivariant Bott}} K^*_G(X\times V\times
V) \to K^*_G(X\times V\times \{0\}),
\]
we also have a natural $R$-module homomorphism $\psi\co K^*_G(X)\to
K^*_{G,-}(X)$.  The composite $\varphi\circ \psi$ is the product
with the element of $R$ associated to $R\cong K^0_G(\pt)
\xrightarrow{\text{Bott}} K^0_G(V_\bC)\to
K^0_G(\{0\})= R$ coming from the inclusion $\{0\}\hookrightarrow \bC$,
where $V_\bC$ is the complexification of $V$ ($\bC$ with the action of
$G$ by multiplication by $-1$). This composite is $1-t$ (see
\cite[\S3]{MR0234452}) , and since  $\psi\circ \varphi$
is the same thing (except applied to $X\times V$ instead of to $X$),
we have proved the following.
\begin{proposition}
\label{prop:definvar}
Let $G$ be the cyclic group of two elements. To any $G$-space there
$X$ is naturally associated a diagram
\[
\bK^*_G(X)\co
\xymatrix{K^*_G(X) \ar@/^/[r]^\psi & K^*_{G,-}(X) \ar@/^/[l]^\varphi,}
\]
where the maps $\varphi$ and $\psi$ preserve the $\bZ/2$-grading and
the composite in either order is multiplication by $1-t$.
\end{proposition}
\begin{examples}\leavevmode
\label{ex:invar}
\begin{enumerate}
\item If $X=\pt$, then $K^*_G(X)=R$ (concentrated in degree $0$) and
  $K^*_{G,-}(X) \cong R/J \cong I$ (concentrated in degree $0$). The
  map $\varphi$ is the inclusion $I \hookrightarrow R$. The map $\psi$
  is the projection $R \twoheadrightarrow R/J$.
\item If $X=V$ ($\bR$ with the sign representation of $G$), then
  $\bK^*_G(V)$ is the same as $\bK^*_G(\pt)$, but with the two
  $R$-modules  interchanged. 
\item If $X=G$ with the simply transitive action of $G$, then
  $K^*_G(X)\cong R/I$ (concentrated in degree $0$) and $K^*_{G,-}(X)
  =K^*_G(G\times V) \cong K^*(V)\cong R/I$ (concentrated in degree
  $1$). The connecting maps are both necessarily $0$.
\end{enumerate}
\end{examples}
Note of course that $\bK^G_*(A)$ can be defined for a $G$-{\Ca} $A$ in
exactly the same way.

The following proposition is a precursor of the main theorem in the next
section. 
\begin{proposition}
\label{prop:equivandnon-equiv}
Let $G$ be a cyclic group of two elements and let $X$ be a locally
compact $G$-space.  Then there is a natural $6$-term exact sequence
\[
\xymatrix{
K^1(X) \ar[r] & K^0_{G,-}(X) \ar[r]^\varphi & K^0_G(X) \ar[d]^f \\
K^1_G(X) \ar[u]^f & K^1_{G,-}(X) \ar[l]^\varphi  & K^0(X), \ar[l]}
\]
where the vertical arrows marked $f$ on the left and right are the
forgetful maps 
from equivariant to non-equivariant $K$-theory. The same {\lp}with the
indices lowered{\rp} holds similarly for $G$-{\Ca}s. 
\end{proposition}
\begin{proof}
We use the fact that if $V$ is the sign representation of $G$ as
above, then $V\smallsetminus \{0\}\cong \bR\times G$
(equivariantly). Here $\bR$ carries the trivial $G$-action but $G$
acts simply transitively on itself.  Thus we get an induced long exact
sequence
\[
\cdots \to K^*_G(X\times \bR\times G) \to K^*_G(X\times V) \to K^*_G(X \times
\{0\}) \to \cdots.
\]
Here the middle group is $K^*_{G,-}(X)$ and the map to $K^*_G(X\times
\{0\})$ is what we defined to be $\varphi$.  The group on the left is
isomorphic to the non-equivariant $K$-group $K^*(X\times \bR)\cong
K^{*+1}(X)$. It remains to show that the connecting map $K^*_G(X \times
\{0\}) \to K^{*+1}(X\times \bR)\cong K^*(X)$ is the forgetful map
$f$. This follows by naturality of products from the fact that it's
true for $X=\pt$ (which one can check from the exact sequence and the
identification of $K^0_G(X)$ with $R$ and of $K^0(X)$ with $R/I$).
\end{proof}
\begin{corollary}
\label{cor:equivandnon-equiv}
If $G$ is a cyclic group of two elements and $X$ is a locally compact
$G$-space, then there is a short exact sequence
\[
0 \to \coker\varphi \to K^*(X) \to \ker\varphi \to 0,
\]
where $\varphi$ is as in Proposition \ref{prop:definvar}.
\end{corollary}
\begin{proof}
This is just a restatement of the exactness in Proposition
\ref{prop:definvar}. 
\end{proof}

\section{The main theorem and its proof}
\label{sec:main}

In this section $G$ will always denote a cyclic group of order $2$. 

Because of Theorem \ref{thm:Phillips}, as well as the fact that an
$R$-module is completely determined by its localizations at maximal
ideals $\fp$ of $R$, to complete the problem of getting a K\"unneth
Theorem for $K^*_G$ (for $G$-spaces) or $K_*^G$ (for $G$-algebras) we
just need to compute $\bK^*_G(X\times Y)_{\fp}$ in terms of $\bK^*_G(X)_{\fp}$
and  $\bK^*_G(X)_{\fp}$ (and similarly for algebras in a suitable bootstrap
category) for maximal ideals $\fp = (I, p)$, $p$ a prime. (Once again,
see Figure \ref{fig:specR}.)

After localization at $\fp = (I, p)$, an interesting thing happens:
since $1-t$ lies in the kernel of the localization map $R\to R_{\fp}$,
$\psi\circ\varphi$ and $\varphi\circ\psi$ are both $0$. But something
much stronger is true.
\begin{proposition}
\label{prop:phivanish}
Let $G$ be a cyclic group of order $2$ and let $\fp = (I,
p)\triangleleft R=R(G)$, $p$ a prime. Then for any $G$-{\Ca} $A$,
$\varphi\co K_*^{G,-}(A) \to K_*^G(A)$ and $\psi\co
K_*^G(A) \to K_*^{G,-}(A)$ vanish after localization at $\fp$. 
\end{proposition}
\begin{proof} It's enough to treat one of $\varphi$ and $\psi$ since
each can be obtained from the other by replacing $A$ by $A\otimes
C_0(V)$. (As usual, $V$ here denotes the sign representation of $G$ on
$\bR$.) Furthermore, by the usual tricks with suspensions and
unitalizations, we can restrict attention to $K_0$ and assume $A$ is
unital. By \cite[\S11.3]{MR1656031}, any class in $K_0(A)$ comes from
a $G$-invariant projection $p$ in $\End(W)\otimes A$, $W$ a
finite-dimensional $G$-module (and thus of the form $\bC^n\oplus
V^m_{\bC}$). Such a projection $p$ defines a $G$-homomorphism
$\bC \to \End(W)\otimes A$.  By functoriality of $\psi$, we get a
commutative diagram
\[
\xymatrix{K_0^G(\bC)_{\fp} = R_{\fp} \ar[r]^\psi \ar[d]^p 
& K_0^{G,-}(\bC)_{\fp} = 0 \ar[d]^{p \otimes 1}\\
K_0^G(\End(W)\otimes A)_{\fp} \ar[r]^(.47)\psi & K_0^{G,-}(\End(W)\otimes A)_{\fp},}
\]
which shows that $\psi([p])=0$ after localization at $\fp$.
\end{proof}

Because of Proposition \ref{prop:phivanish}, we can ignore $\varphi$
and $\psi$ after localization at $\fp$ and treat
$\bK_G^*(X)_{\fp}$ as a $\bZ/2$-graded $R_{\fp}$-module,
by putting $K^0_{G,-}(X)_{\fp}\oplus K^{-1}_{G}(X)_{\fp}$ in odd
degree and $K^0_{G}(X)_{\fp}\oplus K^{-1}_{G,-}(X)_{\fp}$ in even degree.  Our main
result is suggested by the following reformulation of Corollary
\ref{cor:equivandnon-equiv}:
\begin{proposition}
\label{prop:KunnforG}
Let $G$ be a cyclic group of order $2$ and let $X$ be a locally
compact $G$-space. Let $\fp = (I, p)\triangleleft R$, $p$ a prime.
Then there is a filtration on a direct sum of two
copies of $K^*(X)_{\fp}\cong K^*_G(X\times G)_{\fp}$ for which the
associated graded module is $\bK_G^*(X)_{\fp} \otimes \bK_G^*(G)_{\fp}$.
\end{proposition}
\begin{proof}
By Example \ref{ex:invar}(3), $\bK_G^*(G)_{\fp}$ contains two copies of 
$R_{\fp}$ concentrated in degree $0$. Thus tensoring with $\bK_G^*(X)_{\fp}$
is simply tensoring over $R_{\fp}$ with $R_{\fp}\oplus R_{\fp}$, i.e.,
the operation of ``doubling.''  But by Proposition
\ref{prop:equivandnon-equiv} together with Proposition
\ref{prop:phivanish}, $K^*(X)_{\fp}$ is an extension of
$K^{*+1}_{G,-}(X)_{\fp}$ by $K^*_{G}(X)_{\fp}$. So the result follows.
\end{proof}
\begin{remark}
\label{rem:whyI}
It might seem strange that the formulation of Proposition
\ref{prop:KunnforG} is limited to the case of maximal ideals $\fp
\supset I$. If we localize instead at a maximal ideal $\fp$ with
support $G$ (and thus of the form $(J,p)$, $p\ne 2$), then by Example
\ref{ex:invar}(3), $\bK^*_G(G)_{\fp} = 0$. At the same time, $K^*(X)
\cong K^*_G(X\times G)$, and $X\times G$ is a free $G$-space, so by the
Localization Theorem, $K^*_G(X\times G)_{\fp}  = 0$. Thus Proposition
\ref{prop:KunnforG} is still true in this case, but it just says
$0=0$, which obviously isn't very interesting.  By way of further
explanation, Proposition \ref{prop:equivKfree} shows that for a free
compact $G$-space, the augmentation ideal $I$ acts nilpotently on the
equivariant $K$-theory, so only localization at prime ideals
containing $I$ gives anything useful.
\end{remark}

Now we are ready to state and prove the main theorem.
\begin{theorem}
\label{thm:main}
Let $G$ be a cyclic group of order $2$ and let $X$ and $Y$ be second
countable locally
compact $G$-spaces. Let $\fp = (I, p)\triangleleft R$, $p$ a prime.
Then there is a natural short exact sequence of $\bZ/2$-graded
$R_{\fp}$-modules 
\begin{equation}
\label{eq:hyperPhillips}
0 \to \bK^p_G(X)_{\fp} \otimes_{R_{\fp}} \bK^q_G(Y)_{\fp}
\xrightarrow{\omega_{\fp}} \bK^{p+q}_G(X \times Y)_{\fp}
\to \Tor_1^{R_{\fp}}\bigl(\bK^p_G(X)_{\fp}, \bK^{q+1}_G(Y)_{\fp}\bigr)
\to 0.
\end{equation}
Note the similarity to \eqref{eq:Phillips}. 
Similarly, for separable $G$-{\Ca}s $A$ and $B$ 
with $B$ in a suitable ``bootstrap''
category, containing all inductive limits of separable type I
$G$-{\Ca}s and closed under exterior equivalence, equivariant Morita
equivalence, and the ``$2$ out of $3$ property for short exact
sequences,'' we have a natural short exact sequence
\[
0 \to \bK_p^G(A)_{\fp}\otimes_{R_{\fp}}  \bK_{q}^G(B)_{\fp}
\xrightarrow{\omega_{\fp}}  \bK_{p+q}^G(A \otimes B)_{\fp} \to
\Tor_1^{R_{\fp}}\bigl(\bK_p^G(A)_{\fp}, \bK_{q+1}^G(B)_{\fp}\bigr)
\to 0.
\]
\end{theorem}
The method of proof of this theorem is similar to the one used in
\cite{MR849938} and \cite{MR911880}, or in other words, is based on
the method of geometric resolutions.  First we need to see that the
problem that occurred in our previous counterexample to the K\"unneth
Theorem, having $K^*_G(X)=0$ but $K^*_G(X\times Y) \ne 0$ for some
$Y$, cannot recur.
\begin{lemma}
\label{lem:KGzerogivesKun}
Let $\fp = (I, p)\triangleleft R$, $p$ a prime.
Let $A$ be a $G$-{\Ca} with $\bK_*^G(A)_{\fp} = 0$. Then for any
finite $G$-CW complex $Y$, $\bK_*^G(A\otimes C(Y))_{\fp} = 0$. 
\end{lemma}
\begin{proof}
We are assuming $K_*^G(A)=0$ and $K_*^{G,-}(A)=0$, and we need to
show $K^G_*(A\otimes C(Y))=0$. (A similar conclusion for
$K_*^{G,-}(A\otimes C(Y))$ follows upon replacing $A\otimes C(Y)$ by
$A\otimes C(Y)\otimes C_0(V)$, where $V$ as usual is the sign
representation of $G$ on $\bR$.) First we show the result holds
when $Y$  is a single (open) $G$-cell, i.e., either $\bR^n$ or $G\times
\bR^n$ with trivial action on $\bR^n$. (Since we're using open cells
here, $C(Y)$ should be replaced by $C_0(Y)$.) But
$K^G_*(A\otimes C_0(\bR^n))\cong K^G_{*+n}(A)$ and
$K^G_*(A\otimes C_0(G\times \bR^n))\cong K_{*+n}(A)$, to which we can
apply Proposition \ref{prop:equivandnon-equiv}. Now
assume $Z$ is a $G$-space for which we know $K_*^G(A\times C_0(Z))=0$
and $K_*^{G,-}(A\times C_0(Z))=0$, and assume $Y$ is obtained from $Z$
by adding a single equivariant cell, so that $Y\smallsetminus Z$ is 
either $\bR^n$ or $G\times \bR^n$ with trivial action on $\bR^n$.
Applying the $5$-Lemma to the $K$-theory sequences
associated to the equivariant short exact sequence
\[
0 \to A\otimes C_0(Z) \to A\otimes C_0(Y) \to 
A\otimes C_0(Y\smallsetminus Z) \to 0,
\]
we get the result for $Y$. Finally, the lemma in full generality
follows by an induction on the $G$-cells of $Y$.
\end{proof}
\begin{corollary}
\label{cor:fundfam}
Let $\cA_G$ denote the category of separable abelian $G$-{\Ca}s
{\lp}contravariantly equivalent to the category of second countable
locally compact Hausdorff $G$-spaces{\rp} and let $\cC_G$ denote the
smallest category of separable $G$-{\Ca}s containing the separable
type I $G$-{\Ca}s and closed under $G$-kernels, $G$-quo\-tients,
$G$-extensions, equivariant inductive limits, crossed products by
actions of $\bR$ or $\bZ$ commuting with the $G$-action, exterior
equivalence, and $G$-Morita equivalence. 
Let $\fp = (I, p)\triangleleft R$, $p$ a prime.
Let $A$ be a $G$-{\Ca} with $\bK_*^G(A)_{\fp} = 0$. Then for any
$G$-{\Ca} $B$ in $\cA_G$ or $\cC_G$, $\bK_*^G(A\otimes B)_{\fp} = 0$. 
\end{corollary}
\begin{proof}
This follows from Lemma \ref{lem:KGzerogivesKun} by an application of
\cite[Theorem 2.8]{MR849938} or \cite[Theorem 6.4.7]{MR911880}. We
just recall the essence of the argument for the abelian case $B =
C_0(Y)$.  Because of the exact sequence for the pair $(Y, Y^G)$, it is
enough to treat the cases of free and trivial $G$-spaces.  There is
also an easy reduction to the case where the space is compact. But any
compact metrizable space is a countable inverse limit of finite CW
complexes. This plus Lemma \ref{lem:KGzerogivesKun} immediately gives
the case of a trivial $G$-space. If $Y$ is a free compact metrizable
$G$-space, write $Y/G$ as a countable inverse limit of finite CW complexes and
pull back to write $Y$ as a countable inverse limit of free finite
$G$-CW complexes. Since equivariant $K$-theory commutes with
equivariant countable $C^*$-inductive limits, the result follows.
\end{proof}
The next step is to prove Theorem \ref{thm:main} with a projectivity
restriction on $\bK_*^G(A)$.  This first requires a lemma which will also be
needed to do the general case.
\begin{lemma}
\label{lem:geomres}
Let $G$ be a cyclic group of order $2$,  $\fp = (I, p)\triangleleft
R$, $p$ a prime. Let $A$ be a separable $G$-{\Ca}. Let $\cH$ denote an
infinite-dimensional separable complex Hilbert space equipped with a
unitary representation of $G$ that contains both irreducible
representations with infinite multiplicity.
Then there is a commutative
$G$-{\Ca} $C=C_0\bigl(X \amalg (Y\times V)\bigr)$, where $X$ and $Y$
are disjoint unions of finite-dimensional real vector spaces with
trivial $G$-action and $V$ is the sign representation of $G$, and
there is a $G$-homomorphism $\alpha\co C\to A\otimes
C_0(V_{\bC}\oplus\bC)\otimes \cK(\cH)$, such that $\alpha$ induces a surjection
\[
\xymatrix@C+2pc{\bK_*^G(C)_{\fp} \ar@{->>}[r]^(.3){\alpha_*} & 
\bK_*^G(A\otimes C_0(V_{\bC}\oplus \bC)\otimes \cK(\cH))_{\fp} &
\bK_*^G(A)_{\fp} \ar[l]_(.3){\cong} },
\]
where the isomorphism $\bK_*^G(A)_{\fp} \cong \bK_*^G(A\otimes
C_0(V_{\bC}\oplus\bC)\otimes \cK(\cH))_{\fp} $ is the canonical one
coming from equivariant Bott periodicity.

If $\bK_*^G(A)_{\fp}$ is free over $R_{\fp}$, then $C$ and $\alpha$
can be chosen so that $\alpha_*$ is an isomorphism.
\end{lemma}
\begin{proof}
By \cite[Proposition 4.1 and Remark 4.2]{MR849938}, which are
proved using the same trick that appeared in the proof of Proposition
\ref{prop:phivanish}, there is a 
commutative {\Ca}, which we may take to be of the form $C_0(X')$ with
$X'$ a disjoint union of countably many
points and lines on which $G$ acts trivially,
and there is a $G$-map $C_0(X'\times \bR) \to A\otimes
C_0(V_{\bC}\times \bR \times \bR)\otimes \cK(\cH)$ inducing a
surjection on $K_*^G$. If $K_*^G(A)_{\fp}$ is
free over $R_{\fp}$, we can choose the induced map on $K_*^G$
localized at $\fp$ to be an isomorphism. We take $X=X'\times \bR$.
Note by Example \ref{ex:invar}(1) (and its suspension) that
$K^*_{G,-}(X)_{\fp} = 0$.

Similarly, we can apply the same argument to $A\otimes C_0(V)$, and
by definition, $K_*^G(A\otimes C_0(V))_{\fp} = K_*^{G,-}(A)_{\fp}$.
We get a trivial $G$-space $Y'$, again a disjoint union of countably many
points and lines, and a $G$-map $C_0(Y'\times \bR) \to A\otimes
C_0(V\times \bR \times \bR)\otimes \cK(\cH)$ inducing a
surjection on $K_*^G$. Again, if  $K_*^G(A\otimes C_0(V\times
\bR^2))_{\fp} \cong K_*^{G,-}(A)_{\fp}$ is
free over $R_{\fp}$, we can choose the induced map on $K_*^G$
localized at $\fp$ to be an isomorphism.  Take $Y=Y'\times \bR$ and
tensor everything with $C_0(V)$. Recall that $C_0(V)\otimes C_0(V)=
C_0(V_{\bC})$. We now have a $G$-map
$C_0(Y\times V)\to A\otimes C_0(V_{\bC}\oplus\bC)\otimes \cK(\cH)$
inducing a surjection on $K_*^{G,-}$ after localization at $\fp$, and
inducing an isomorphism if $K_*^{G,-}(A)_{\fp}$ is free over
$R_{\fp}$. Since $K^*_{G,-}(Y)_{\fp} = 0$, $K^*_{G}(Y\times V)_{\fp} =
0$. The lemma follows upon assembling everything together.
\end{proof}
\begin{theorem}
\label{thm:mainspecial}
Let $G$ be a cyclic group of order $2$,  $\fp = (I, p)\triangleleft
R$, $p$ a prime. Let $A$ be a separable $G$-{\Ca} with
$\bK^G_*(A)_{\fp}$ free over $R_{\fp}$.  Let $B$ be 
any $G$-{\Ca} in $\cA_G$ or $\cC_G$ {\lp}in the notation of \textup{Corollary
\ref{cor:fundfam}{\rp}}. Then there
is a natural isomorphism $\bK_G(A)_{\fp}\otimes_{\fp} \bK_G(B)_{\fp} \to
\bK_G(A\otimes B)_{\fp}$ 
coming from the pairing $\omega_{\fp}$ discussed in \cite[\S6.1]{MR911880}.
\end{theorem}
\begin{proof}
The theorem is obviously true for $A=\bC$ with trivial action, and it
follows that it is true for $A=C_0(V)$ also, since this has
essentially the same $K$-theory (except for interchange of $K_0^G$
with $K_0^{G,-}$, by Example \ref{ex:invar}(2)).  Note that since we are
localizing at a maximal ideal containing $I$,
$K_*^{G,-}(\bC)_{\fp}=0$, whereas $K_*^G(C_0(V))_{\fp}=0$. The theorem
is also true for $A=C(G)$ by Theorem \ref{prop:KunnforG}, and then
follows for $A=C_0(G\times V)$ or $A=C_0(G\times \bR)$ as well.

Now let's do the general case. Assume $C$ and $\alpha$ are chosen as
in  Lemma \textup{\ref{lem:geomres}} to induce isomorphisms on $K_*^G$
and on $K_*^{G,-}$ after localization at $\fp$.
Let $W$ denote the mapping cone of
$\alpha$. Then we obtain a short exact sequence of $G$-algebras
\begin{equation}
0 \to A \otimes C_0(V_{\bC}\oplus\bC) \otimes \cK(\cH) \to W \to C \to 0,
\label{eq:mappingcone}
\end{equation}
for which the induced long exact sequence in $K_*^G$ localized at
$\fp$ gives an isomorphism $K_*^G(C)_{\fp} \xrightarrow{\cong}
K_*^G(A)_{\fp} $ and an isomorphism $K_*^{G,-}(C)_{\fp} \xrightarrow{\cong}
K_*^{G,-}(A)_{\fp} $.  It follows that $\bK_*^G(W)_{\fp}=0$. By
Corollary \ref{cor:fundfam}, $\bK_*^G(W\otimes B) =0$. But tensoring
with $B$ is exact, since $B$ is nuclear, so we get a short exact
sequence
\[
0 \to A \otimes C_0(V_{\bC}\oplus\bC) \otimes \cK(\cH) \otimes B
\to W\otimes B \to C \otimes B\to 0
\]
from \eqref{eq:mappingcone}.
The long exact sequence now gives an isomorphism
$(\alpha\otimes 1)_*\co \bK_*^G(C \otimes B)_{\fp} \to
\bK_*^G(A \otimes B)_{\fp}$. But by the case already considered
(because of the special structure of the algebra $C$),
\[
\omega_{\fp}\co \bK_*^G(C)_{\fp}\otimes_{R_{\fp}}\bK_*^G(B)_{\fp}
\xrightarrow{\cong}
\bK_*^G(C \otimes B)_{\fp}.
\]
Combining the isomorphisms and a little diagram chase now shows that
\[
\omega_{\fp}\co \bK_*^G(A)_{\fp}\otimes_{R_{\fp}}\bK_*^G(B)_{\fp}
\xrightarrow{\cong}
\bK_*^G(A \otimes B)_{\fp}.\qedhere
\]
\end{proof}
Finally we can prove the main theorem.
\begin{proof}[Proof of Theorem \ref{thm:main}]
We apply Lemma \ref{lem:geomres} and take the mapping cone sequence
\eqref{eq:mappingcone}. Since $\bK_*^G(C)_{\fp} \xrightarrow{\alpha_*} 
\bK_*^G\bigl(A\otimes C_0(V_{\bC}\oplus \bC)\otimes
\cK(\cH)\bigr)_{\fp} \cong \bK_*^G(A)_{\fp}$ is surjective, $R_{\fp}$
is a PID, and 
$\bK_*^G(C)_{\fp}$ is free over $R_{\fp}$, $\bK_*^G(W)_{\fp}$ is also
free over $R_{\fp}$; in fact 
\begin{equation}
\label{eq:resolution}
\bK_*^G(W)_{\fp} \to \bK_*^G(C)_{\fp} \to \bK_*^G(A)_{\fp}
\end{equation}
is a free $R_{\fp}$-resolution of $\bK_*^G(A)_{\fp}$.  To simplify the
notation, let's replace $A$ by $A\otimes C_0(V_{\bC}\oplus \bC)\otimes
\cK(\cH)$; this is harmless since the theorem will be true for 
$A$ if it's true for the latter. As in the proof of Theorem
\ref{thm:mainspecial}, we tensor \eqref{eq:mappingcone} with $B$ and
apply $\bK_*^G$. We now get a commutative diagram with exact columns,
where the first column comes from tensoring \eqref{eq:resolution} with
$\bK_*^G(B)_{\fp}$ and the second column is the long exact sequence in
$\bK_*^G$. The result is as follows:
\[
\xymatrix@R-.5pc{ 
\Tor_1^{R_{\fp}}(\bK_*^G(A)_{\fp}, \rule[-7pt]{0pt}{7pt}\bK_*^G(B)_{\fp})
\ar@{>->}[d] & \bK_{*+1}^G(A \otimes B)_{\fp}\ar[d] \\
\bK_*^G(W)_{\fp} \otimes_{R_{\fp}} \bK_*^G(B)_{\fp}
\ar[r]^(.6){\omega_{\fp}}
_(.6){\cong} \ar[d]  &\bK_*^G(W\otimes B)_{\fp} \ar[d]\\
\bK_*^G(C)_{\fp} \otimes_{R_{\fp}} \bK_*^G(B)_{\fp}
\ar[r]^(.6){\omega_{\fp}}
_(.6){\cong} 
\ar@{->>}[d]^{\alpha_*\otimes 1}  &\bK_*^G(C\otimes B)_{\fp} \ar[d]^{(\alpha\otimes 1_B)_*}\\
\bK_*^G(A)_{\fp} \otimes_{R_{\fp}} \bK_*^G(B)_{\fp} \ar[r]^(.6){\omega_{\fp}}
&\,\bK_*^G(A\otimes B)_{\fp} .
}
\]
The fact that the middle horizontal arrows are isomorphisms follows
from Theorem \ref{thm:mainspecial}.

We finish the proof with a diagram chase. If a class in
$\bK_*^G(A\otimes B)_{\fp}$ maps to $0$ in $\bK_{*-1}^G(W\otimes
B)_{\fp}$, then it comes from $\bK_*^G(C\otimes B)_{\fp}$ and thus
lies in the image of $\bK_*^G(A)_{\fp} \otimes_{R_{\fp}}
\bK_*^G(B)_{\fp}$ under $\omega_{\fp}$. Furthermore, if a class
$c\in \bK_*^G(A)_{\fp} \otimes_{R_{\fp}} \bK_*^G(B)_{\fp} $ 
maps to $0$ under $\omega_{\fp}$,
then $c = (\alpha_*\otimes 1)(d)$ for some $d\in \bK_*^G(C)_{\fp}
\otimes_{R_{\fp}} \bK_*^G(B)_{\fp}$, and chasing the diagram shows
that $d$ came from $\bK_*^G(W)_{\fp} \otimes_{R_{\fp}}
\bK_*^G(B)_{\fp}$, which shows that $c=0$. Thus $\omega_{\fp}$ is
injective, and its image is the same as the image of 
$\bK_*^G(C\otimes B)_{\fp}$ in $\bK_*^G(A\otimes B)_{\fp}$, which is precisely
the kernel of the boundary map to $\bK_{*-1}^G(W\otimes B)_{\fp}$.

Now consider the cokernel of $\omega_{\fp}$. By the above discussion,
this is identified with the image of $\bK_*^G(A\otimes B)_{\fp}$ in
$\bK_{*-1}^G(W\otimes B)_{\fp}$, which is the kernel of the map
$\bK_{*-1}^G(W\otimes B)_{\fp}\to \bK_{*-1}^G(C\otimes B)_{\fp}$,
which by the diagram again be identified with
$\Tor_1^{R_{\fp}}(\bK_*^G(A)_{\fp}, \bK_{*-1}^G(B)_{\fp})$. That completes
the proof of the main theorem.
\end{proof}

\bibliographystyle{amsplain}
\bibliography{EquivKunneth}

\end{document}